\documentclass{article}

\usepackage{graphicx}
\usepackage{camnum}
\usepackage{amsmath}
\usepackage{tikz-qtree}
\usepackage{tikz,pgfplots}
\usepackage{comment}
\usepackage{tikz}
\usepackage{amssymb}
\usepackage{pgf, tikz}
\usetikzlibrary{arrows, automata}
\usetikzlibrary{trees}
\usepackage{fancybox}
\usepackage[a4paper, total={7in, 8in}]{geometry}
\usepackage{mathtools}
\usepackage{booktabs}
\usepackage{caption}
\usepackage{subcaption}
\usepackage{caption}

\allowdisplaybreaks

\newcommand{\ee}{{\mathrm e}}
\newcommand{\ii}{{\mathrm i}}

\newcommand{\ad}{\mathrm{ad}\,}

\definecolor{darkgreen}{rgb}{0.0, 0.42, 0.24}

\newtheorem{lem}{Lemma}[section]
\newtheorem{thm}{Theorem}[section]
\newtheorem{remark}{Remark}[section]
\newtheorem{defi}{Definition}[section]
\newtheorem{cor}{Corollary}[section]
\newtheorem{exam}{Example}[section]
\newtheorem{assumption}{Assumption}

\begin{document}

\title{Asymptotic expansions for the linear PDEs with oscillatory input terms;\\
 Analytical form and error analysis} 
\author{
Karolina Kropielnicka\\
Institute of Mathematics\\
Polish Academy of Sciences\\
A. Abrahama 18, 81-825 Sopot\\
Poland
 \and
Rafa{\l}  Perczy\'{n}ski\\
Institute of Mathematics\\
University of Gda\'{n}sk\\
W. Stwosza 57, 80-308 Gda\'{n}sk\\
Poland
}
                                                                                                                                                                                                                                                                                       
\thispagestyle{empty}
\maketitle

\begin{abstract}
Partial differential equations with highly oscillatory input terms are hardly ever solvable analytically and their  numerical treatment is difficult. Modulated Fourier expansion used as an {\it ansatz} is a well known and extensively investigated tool in asymptotic numerical approach for this kind of problems. Although the efficiency of this approach has been recognised, its error analysis has not been investigated rigorously for general forms of linear PDEs. In this paper, we start such kind of investigations for a general form of linear PDEs with an input term characterised by a single high frequency. More precisely we derive an analytical form of such an expansion and provide a formula for the error of its truncation. Theoretical investigations are illustrated by computational simulations.
\end{abstract}

%-----------------------INTRODUCTION----------------------------------------------
\section{Introduction}

In this paper we are concerned with the linear PDEs featured by highly oscillatory input term $f(x,t)$:
\begin{align}
	\label{eq:1.5}
	&\partial_tu(x,t)=\mathcal{L}u(x,t)+f(x,t)u(x,t),\qquad t\in [0,t^\star],\ x\in\Omega\subset\mathbb{R}^m,\\ \nonumber
	&u(x,0)=u_0(x),
\end{align}
where $\Omega$ is a bounded subset of $\mathbb{R}^m$ with smooth boundary $\partial \Omega$ while $\mathcal{L}$ is a differential operator of degree $p$ acting  in the following way 
\begin{align} \label{diff_operator}
	\mathcal{L}v(x) = \sum_{|\textbf{p}|\leq p}^{}a_{\textbf{p}}(x)\frac{\partial^{|\mathbf{p}|}v(x)}{\partial x_1^{p_1}\dots\partial x_m^{p_m}},  \quad \mathbf{p} = (p_1,\dots, p_m), \ p> 0,\ \quad  x\in \Omega.
\end{align}

The type of the boundary conditions does not need to be specified for the purpose of this paper and depends on the applications. We require, however,  the whole initial boundary problem of considerations to be well-posed. The main assumption is the oscillatory in time nature of the input term, more precisely in this manuscript we assume that
\begin{equation}\label{function_f}
	f(x,t) =\alpha(x)\ee^{\ii\omega t}, \quad \omega \gg 1,
\end{equation}
where parameter $\alpha(x)$ is independent of $\omega$.
Let us observe that the presented theory is also applicable to function $f(x,t) =\beta(x)+\alpha(x)\ee^{\ii\omega t}$, which can be obtained by incorporating (also independent of $\omega$) $\beta(x)$ into the linear operator $\mathcal{L}$.

The aim of this paper is to express the solution of (\ref{eq:1.5}) as an asymptotic series 
\begin{equation}\label{expansion}
	u(x,t) \thicksim p_{0,0}(x,t) +  \sum_{r=1}^{\infty}\frac{1}{\omega^r}\sum_{n=0}^{\infty}p_{r,n}(x,t)\ee^{\ii n\omega t},
\end{equation}
which is also known as {\em modulated Fourier expansion}.
In this paper we make the first step towards the derivation  of an asymptotic expansion of the solution of (\ref{eq:1.5}) for a more general highly oscillatory function $f(x,t)$ of the following form
\begin{equation}\label{ideal_function_f}
	f(x,t) = \sum_{n=-N}^{N}\alpha_n(x,t)\ee^{\ii \omega_n t}, \quad |\omega_n| \gg 1, \quad N \in \mathbb{N},
\end{equation}
where $\alpha_n(x,t)$ are independent of $\omega_m$, $n,m\in\mathbb{N}$,  and where possibly occurring resonances will require additional attention.

The numerical-asymptotic approach based on (\ref{expansion}) for problem of type (\ref{eq:1.5}) with (\ref{ideal_function_f}) has been widely applied to many differential problems. For example paper \cite{CDI_2009} engages this methodology for problems in electronic engineering, \cite{CDI_2010} for systems of differential equations with extrinsic oscillation, \cite{CDIK_2012} for ODEs with time delay, while wave-like problems were threated in \cite{FS_2014},  \cite{CIKS_2019} and \cite{CKLP_2021}. This methodology bases on establishing the ansatz solution of type (\ref{expansion}) for the investigated problem. Incorporating the ansatz to the differential equation allows for establishing recurrent or simple, non-oscillatory differential relations for the coefficients $p_{r,n}(x,t)$. Final calculations of the coefficients often require numerical treatment of simple, non-oscillatory PDEs. For this reason methods based on the described approach are called {\it numerical-asymptotic} solvers. Numerical-asymptotic methods stem from the theory described in \cite{HLW_2006}, where the theoretical fundaments has been presented for slightly simpler case of expansions (where summation over powers of magnitude $\omega^{-r}$ is not taken into account).\\

The novelty of the paper is twofold: We derive (\ref{expansion}) purely analytically (presenting the formulas for $p_{r,n}$) and we provide the formula for the error of the approximation. Unlike in the cited publications we do not assume that ansatz (\ref{expansion}) solves (\ref{eq:1.5}) but derive it making use of Duhamel formula, Neumann series, and of integration by parts.

To present the error of the approximation we observe that in practice instead of an infinite series (\ref{expansion}) we rather resort to partial sums:
\begin{equation}\label{partial_expansion}
	u(x,t) = p_{0,0}(x,t) +  \sum_{r=1}^{R}\frac{1}{\omega^r}\sum_{n=0}^{N}p_{r,n}(x,t)\ee^{\ii n\omega t}+\mathcal{R}_{R,N}, \qquad t\in [0,t^\star],\ x\in\Omega\subset\mathbb{R}^m,
\end{equation}
and are interested in expression $\mathcal{R}_{R,N}$. According to our knowledge no papers concerned with the numerical-asymptotic solvers based on modulated Fourier expansions provide neither formula $\mathcal{R}_{R,N}$ nor its error bounds. All that has been concluded by now is that the leading error term  (not the error!) of such truncation scales like $\omega^{-R-1}$.  This is however misleading, because the error constant depends on the differential operator $\mathcal{L}$, on the initial condition $u_0(x)$ and on the parameters $\alpha_n (x,t)$. Let us emphasize that the error constant of the leading error term has never been investigated and the form of dependence on the aforementioned parameters has not been known. In this manuscript we provide analytical formulas not only for functions $p_{r,n}(x,t)$, but also for the error term $\mathcal{R}_{R,N}$. In contrast to previous research we investigate the whole error term of the truncation (not only the leading error term). Moreover knowledge of the exact formula $\mathcal{R}_{R,N}$ allows for the realistic estimate of the error. This is important, because it is widely known that modulated Fourier expansions do not have to be convergent. This may happen once error constant $C_R$ of the error term $\mathcal{R}_{R,N}$ exceeds the magnitude of $\omega^{-R-1}$. For this particular reason the knowledge of the formula $\mathcal{R}_{R,N}$ is so important.

Modulated Fourier expansions of type (\ref{expansion}) have been successfully applied to the problems, where oscillatory nature of solutions has intrinsic origin rather than extrinsic, see for example \cite{HCH_2001,GHCH_18,WZ_21,WZ_23}. For obvious reasons our study does not cover these problems.

In Section \ref{Duhamel_Neumann} we present solution of the underlying problem with the  Duhamel formula and will present it as a series of multivariate integrals. Section \ref{Series} is devoted to the asymptotic expansion of these multivariate integrals derived. The main theorem on the asymptotic expansion of the  solution to (\ref{eq:1.5}) is presented in Section \ref{Main_result}. Numerical simulations are presented in Section \ref{Numerics}.

%---------------------------------Duhamel_Neumann-------------------------------
\section{The solution as a series of the multivariate oscillatory integrals}\label{Duhamel_Neumann}

In this section we will show that function $u\in C^1([0,t^\star ),H^{p}(\Omega))$ solving problem (\ref{eq:1.5}) with $u_0\in H^p(\Omega)$ can be represented as an infinite sum of multivariate highly oscillatory integrals  which, as will be mentioned in Remark \ref{scaling} and later claimed in Theorem \ref{integral_AB}, scale like $\mathcal{O}(\omega^{-d}), \ d=0,1,2,\dots$.

Let us recall that for the sake of generality of applications of the presented methodology we do not impose any boundary conditions, but require well possedness of the problem under considerations. The following assumption is necessary for reasonings presented in this section.

\begin{assumption}
	\label{main_assumption}
	We assume that problem (\ref{eq:1.5}) equipped with suitable boundary conditions is well posed and seek  solutions $u[t]:=u(\cdot,t) \in C^1([0,t^\star ],H^{p}(\Omega))$ while  $u_0\in H^p(\Omega)$. Moreover we assume that the linear, differential operator  $\mathcal{L}: D(\mathcal{L}) \supseteq  H^{p}(\Omega) \rightarrow L^2(\Omega)$ defined with (\ref{diff_operator}) is a generator of $C_0$ (strongly continuous) semigroup $e^{t\mathcal{L}}: L^2(\Omega)\rightarrow L^2(\Omega)$, that  $f[t]:=f(\cdot,t)\in C^1([0,t^\star ),L^\infty(\Omega))$ and that the functions $ a_{\textbf{p}}\in C_b^\infty(\Omega)$, $|\textbf{p}|\leq p$, meaning that their all derivatives are bounded on $\Omega$
\end{assumption}

\begin{remark}
	Some boundary conditions may enforce additional assumptions on the domain of the operator $\mathcal{L}$. For example zero boundary conditions would  require  $H^{p}(\Omega)\cap H^{1}_0(\Omega)$  instead of $H^{p}(\Omega)$ in Assumption \ref{main_assumption}. Careful changes of the domains of operators and/or coefficients, however, would not affect the suggested methodology.
\end{remark}

Having Assumption \ref{main_assumption} satisfied, application of the Duhamel principle (also known as variation-of-constants) to equation (\ref{eq:1.5}) results in its equivalent integral form
\begin{eqnarray}\label{eq:Duhamel}
	u[t] = \ee^{t\mathcal{L}}u_0+ \int_{0}^{t}\ee^{(t-\tau)\mathcal{L}}f[\tau]u[\tau]\D\tau.
\end{eqnarray}

Let us define space $Y:=C([0,t^\star];L^2(\Omega))$ with standard norm $\|v\|_Y=\max_{0\leq t\leq t^\star}\|v[t]\|_{L^2(\Omega)}$. Then for function $f$ and operator $\mathcal{L}$ satisfying Assumption \ref{main_assumption} we can define operator $K(u)[t]:Y\rightarrow Y$ as an integral
\begin{equation}\label{K}
	K(u)[t] = \int_{0}^{t}\ee^{(t-\tau)\mathcal{L}}f[\tau]u[\tau] \D\tau, \qquad 0\leq t\leq t^\star,
\end{equation}
and analyse its multiple compositions
\begin{align}\label{nested_int}
	K^d (u)[t] = \int_{0}^{t}\ee^{(t-\tau_1)\mathcal{L}}f[\tau_1]\int_{0}^{\tau_1}\ee^{(\tau_1-\tau_2)\mathcal{L}}f[\tau_2]\dots\int_{0}^{\tau_{d-1}}\ee^{(\tau_{d-1}-\tau_d)\mathcal{L}}f[\tau_d]u[\tau_d]\D \tau_{d} \dots \D \tau_1.
\end{align}
Highly oscillatory nature of function $f$, see definition (\ref{function_f}), indicates that these are $d$ times nested highly oscillatory integrals.
Obviously equation (\ref{eq:Duhamel}) can be expressed as
\begin{equation}\label{eq:Symbolic}
	u[t] = \ee^{t\mathcal{L}}u_0+Ku[t], \quad t \in [0,t^\star],
\end{equation}
and, as often happens, solved in means of the iterative approximations
\begin{equation}\label{iterations}
	u^{[n]}[t] = \ee^{t\mathcal{L}}u_0+Ku^{[n-1]}[t],\ {\rm where}\ u^{[0]}=\ee^{t\mathcal{L}}u_0.
\end{equation}
We can easily observe that formula (\ref{iterations}) leads to the well recognised and widely applied truncated sums of Neumann series, also known as Dyson series or {\em Time Ordering } operator
\begin{equation}\label{Dyson}
	u^{[n]}[t] = \sum_{d=0}^{n }K^d \ee^{t\mathcal{L}}u_0,\ {\rm where}\ K^0={\rm Id}.
\end{equation}

The following theorem is crucial for further investigations as it not only justifies forthcoming expansions but also clearly establishes error bounds for approximation in means of Neumann series for the particular form of (\ref{eq:1.5}), thus of (\ref{eq:Duhamel}) as they are equivalent in our setting.

\begin{thm} \label{Neumann_f(x,t)}
	Let us make  Assumption \ref{main_assumption}. Then the elements of an infinite  sequence $\left\{u^{[n]}\right\}_{n=0}^\infty \in Y$ defined as in (\ref{Dyson}) with the operator $K$ given with (\ref{K}) converge in norm $\|\,\cdot\,\|_Y$ to the unique solution $u^{\star}$  of (\ref{eq:1.5}) as $n\rightarrow\infty$.  
	Moreover, the following estimate holds
	\begin{equation}\label{estimate_u_star}
		\qquad \|u^\star- u^{[n]}\|_Y\leq C_{t^\star}\|u_0\|_{L^2(\Omega)}\exp\Big(C_{t^\star}  \|f\|_Y {t^\star}\Big) \left\|K^{n+1}\right\|_Y,  \\
	\end{equation}
	where $ C_{t^\star}=\max_{0\leq t\leq {t^\star}}\|e^{t\mathcal{L}}\|_{L^2(\Omega)} $.
\end{thm}

\begin{proof}%\kern-1em{\it of Theorem}  \ref{Neumann_f(x,t)}
	Let us start with the observation that there exists $C_{t^\star}$ such that $\|\ee^{t\mathcal{L}}\|_{L^2(\Omega)}\leq C_{t^\star}$, $t\in[0,{t^\star}]$, and that 
	
	\begin{eqnarray}\label{number_k}
		\|K^n\|_Y&\leq& \left(C_{t^\star}\|f\|_Y\right)^{n}\int_{0}^{t^\star}\int_{0}^{\tau_1}\dots\int_{0}^{\tau_{n-1}}\D\tau_{n}\dots\D\tau_1 \leq  \frac{(C_{t^\star}\|f\|_Y {t^\star})^{n}}{n!}.
	\end{eqnarray}
	This means that $\sum_{d=0}^{\infty}\left\|K^d\right\|_Y \leq \exp\Big(C_{t^\star}  \|f\|_Y {t^\star}\Big)$ and  that $\left\{u^{[n]}\right\}_{n=0}^\infty \in Y$ forms a Cauchy sequence with its limit defined as
	
	\begin{eqnarray}\label{u_star}
		u^\star[t]:=\lim_{n \to \infty } u^{[n]}[t]=\sum_{d=0}^{\infty}K ^d \ee^{t\mathcal{L}}u_0.
	\end{eqnarray}
	Now we can notice that $u^\star$ solves (\ref{eq:1.5}). Indeed, since $\lim_{n\rightarrow\infty}\|K^{n}\|_Y=0$ we obtain
	$$
	(I-K)u^\star[t]=(I-K)\sum_{d=0}^{\infty}K^d\ee^{t\mathcal{L}}u_0=\left(I-\lim_{n\rightarrow\infty}K^{n}\right)\ee^{t\mathcal{L}}u_0=\ee^{t\mathcal{L}}u_0,$$
	which coincides with the fixed point of (\ref{eq:Symbolic}) thus solves (\ref{eq:Duhamel}) and (\ref{eq:1.5}).
	Observation that 
	\begin{equation}\label{difference}
		u^\star[t]-u^{[n]}[t]=K^{n+1}\sum_{d=0}^{\infty}K^d\ee^{t\mathcal{L}}u_0
	\end{equation}
	together with (\ref{number_k}) results  in the estimate (\ref{estimate_u_star}) which completes the poof.
\end{proof}

\begin{remark}\label{scaling}
	Estimates (\ref{estimate_u_star}) and  (\ref{number_k}) lead to
	\begin{equation}\label{rate}
		\left\|u^{\star}-u^{[n]}\right\|_Y\leq  C_{t^\star} \|u_0\|_{L^2(\Omega)} \exp\Big(C_{t^\star} \|f\|_Y {t^\star}\Big)  \frac{\left(C_{t^\star} \|f\|_Y {t^\star}\right)^{n+1}}{ (n+1)!}
	\end{equation}
	which obviously specifies certain rate of convergence. Given, however that $f$ is highly oscillatory in time, i.e. $f(x,t) = \ee^{\ii\omega t}\alpha(x),\ \omega \gg 1$ this estimate is suboptimal. It is well known that integrals of highly oscillatory functions are of small magnitude. To illustrate it let us consider the integral of a continuously differentiable function $g(s)=c(s)\ee^{\ii\omega s}$, $s\in[0,t]$. One can easily observe that
	$$
	\left| \int_0^t c(s)\ee^{\ii\omega s}\D s\right|\leq \bar{C} \min\{t,\omega^{-1}\},
	$$
	where $\bar{C}$ depends on the coefficient $c(s)$ and on its derivative.
	
	Following this argumentation and observing that estimate (\ref{estimate_u_star}) utilises multivariate integral of highly oscillatory functions, see (\ref{nested_int}), we can deduce that under stronger assumptions on our problem the estimate of the rate of convergence in (\ref{rate}) may be improved to
	$$
	\left\|u^{\star}-u^{[n]}\right\|\leq   C^n \cdot \omega^{-n-1}.
	$$
	One needs to keep in mind, however that  the constant is independent of $\omega$ but might get dangerously large (even larger then $\omega^{-n-1}$ !) as it depends on other few factors, namely $C^n(t^\star,\mathcal{L},u_0,\alpha)$. For this reason we are interested in the formula describing the error of truncation.
\end{remark}

%-----------------------------Multivariate_oscillatory---------------------------

\section{Multivariate oscillatory integral as an asymptotic expansion}\label{Series}

In this section we are concerned with the operator (\ref{K}) where $u=\ee^{t\mathcal{L}}u_0$ and $f(x,t)=\alpha(x)\ee^{i\omega t}$, $\omega\gg 1$. To distinguish this form we will use letter $T$ instead of $K$:
\begin{equation}\label{T}
	T^d\ee^{t\mathcal{L}}u_0= \int_{0}^{t}\ee^{(t-\tau_1)\mathcal{L}}\alpha\ee^{\ii\omega \tau_1}\int_{0}^{\tau_1}\ee^{(\tau_1-\tau_2)\mathcal{L}}\alpha\ee^{\ii\omega \tau_2}\dots\int_{0}^{\tau_{d-1}}\ee^{(\tau_{d-1}-\tau_d)\mathcal{L}}\alpha\ee^{\ii\omega t_{d}}\ee^{\tau_d\mathcal{L}} u_0 \D \tau_{d} \dots \D \tau_1
\end{equation}

We are interested in an asymptotic expansion of integrals $T^d\ee^{t\mathcal{L}}u_0,  d=1,\dots n$. More specifically for given $d$ we will present $T^d\ee^{t\mathcal{L}}u_0 $ as a finite series and an error term related to the asymptotic expansion. This will lead us towards the expansions of  $n$-th approximations of  (\ref{eq:Duhamel}). Indeed, it is enough to observe that $u^{[n]}$ is defined, see (\ref{Dyson}), as an infinite sum of integrals  $T^d\ee^{t\mathcal{L}}u_0$. Let start with necessary assumption and auxiliary lemmas to which we will refer to in the further part of this manuscript.

\begin{assumption}
	\label{stronger_assumption}
	We assume that problem (\ref{eq:1.5}) equipped with suitable boundary conditions is well posed and seek  solutions $u[t]:=u(\cdot,t) \in C^n([0,t^\star ],H^{np}(\Omega))$ while  
	$u_0\in H^{np}(\Omega)$, where $n\in\mathbb{N}$ is a fixed number. Moreover we assume that the linear, differential operator  $\mathcal{L}: D(\mathcal{L}) \supseteq  H^{np}(\Omega) \rightarrow L^2(\Omega)$ defined with (\ref{diff_operator}) is a generator of a differential (in strong operator topology) semigroup $e^{t\mathcal{L}}: L^2(\Omega)\rightarrow L^2(\Omega)$, that  
	$\alpha \in W^{np,\infty}(\Omega)$ 
	
	and that the functions $ a_{\textbf{p}}\in C_b^\infty(\Omega)$, $|\textbf{p}|\leq p$, meaning that their all derivatives are bounded on $\Omega$
\end{assumption}

\begin{lem}\label{AD}
	Let us make Assumption \ref{stronger_assumption}. The  $k$-th derivative of  expression $\ee^{(t-\tau)\mathcal{L}}\alpha\ee^{\tau\mathcal{L}}$ can be presented in the following way
	\begin{eqnarray*}
		&&	\partial_\tau^k\left(\ee^{(t-\tau)\mathcal{L}}\alpha\ee^{\tau\mathcal{L}}\right)= \ee^{(t-\tau)\mathcal{L}}\underbrace{\Big[\dots\big[[\alpha,\mathcal{L}],\mathcal{L}\big],\dots\Big]}_{\mathcal{L} \ appears\ k \ \text{times}}\ee^{\tau\mathcal{L}}=(-1)^k\ee^{(t-\tau)\mathcal{L}}\ad_\mathcal{L}^k\big(\alpha\big)\ee^{\tau\mathcal{L}},\quad k\leq n,
	\end{eqnarray*}
	where $\ad_\mathcal{L}^0\big(\alpha\big)=\alpha, \quad  \ad_\mathcal{L}^{k}\big(\alpha\big)=\big[\mathcal{L}, \ad_\mathcal{L}^{k-1}\big(\alpha\big)\big]$ and $[X,Y]\equiv XY-YX$ is the commutator of $X$ and $Y$.
\end{lem}

\begin{lem}\label{IA0BLemma}
	Let us make Assumption \ref{stronger_assumption}. The following  integral can be expressed as a sum of $n+1$ components scaling like $\omega^{-m},\quad m=1,\ldots n.$
	\begin{eqnarray}\label{IA0B}
		&&\int_{0}^{t}\ee^{\ii\omega\tau}\ee^{(t-\tau)\mathcal{L}}\alpha\ee^{\tau\mathcal{L}}u_0\D\tau=\sum_{m=0}^{n-1}\frac{1}{(\ii\omega)^{m+1}}\Big[\ee^{\ii\omega t}\ad_\mathcal{L}^m\big(\alpha\big)\ee^{t\mathcal{L}}-\ee^{t\mathcal{L}}\ad_\mathcal{L}^m\big(\alpha\big)\Big]u_0\\ \nonumber
		&&\hspace{3.8cm}+\frac{1}{(\ii\omega)^n}\int_{0}^{t}\ee^{\ii\omega\tau}\ee^{(t-\tau)\mathcal{L}}\ad_\mathcal{L}^n\big(\alpha\big)\ee^{\tau\mathcal{L}}u_0\D\tau \nonumber
	\end{eqnarray}
\end{lem}

Lemma \ref{AD} can be easily proved by induction while Lemma \ref{IA0BLemma} by $n$-timeas integration by parts. Let us observe that in case $u_0\in H^{(n+1)p}(\Omega)$ and $ \alpha \in W^{(n+1)p,\infty}(\Omega)$, the last component of expansion (\ref{IA0B}) scales like $\omega^{-(n+1)}$. Indeed, the last term can be integrated by parts again.

To derive the formula for the asymptotic series of the multivariate highly oscillatory integral $T^d\ee^{t\mathcal{L}}u_0$,  we will heavily utilise Lemma \ref{IA0BLemma}. We will start with the asymptotic expansion of the innermost integral, then we will expand asymptotically each successive integrals up to the  outermost. Integration by parts of each integral will result in appearance of the two terms $\sum_{m=0}^{n-1}\frac{1}{(\ii\omega)^{m+1}}\ee^{\ii\omega t}\ad_\mathcal{L}^m\big(\alpha\big)\ee^{t\mathcal{L}}u_0$ and $\sum_{m=0}^{n-1}\frac{-1}{(\ii\omega)^{m+1}}\ee^{t\mathcal{L}}\ad_\mathcal{L}^m\big(\alpha\big)u_0$. As a consequence, the number of components in the series of expansion of $T^d\ee^{t\mathcal{L}}u_0$ will grow exponentially in $d$, therefore to facilitate the notation we introduce the following definitions.

\begin{defi}\label{multindex}
	Let us define multi-index ${\bf k^s}=(k_1,\ldots,k_s)\in \mathbb{N}^s$, and understand that $|{\bf k^s}|=|{\bf k^s}|_1=k_1+\ldots+k_s$.
\end{defi}
\begin{defi}\label{N}
	Let us make Assumption \ref{stronger_assumption} and define operator $\mathcal{N}_{k_l}^{k_r}[\mathcal{L},\alpha]$ as follows
	$$
	\mathcal{N}_{k_l}^{k_r}[\mathcal{L},\alpha]=
	\left\{
	\begin{array}{ll}
		\ad_\mathcal{L}^{k_{r}}\big(\alpha \ad_\mathcal{L}^{k_{r-1}}(\dots  \alpha \ad_\mathcal{L}^{k_{l}}(\alpha))\big), &  l< r\\
		\\
		\ad_\mathcal{L}^{k_l}(\alpha), & l=r\\
		\\
		{\rm Id}, & l > r \\
	\end{array}
	\right.,
	$$
	where $k_l, k_{l+1},\dots, k_{r-1},k_r$,  $l\leq r$ are inputs of multi-index ${\bf k^{s}}$, and {\rm Id} is an identity function on space $W$.
	For clarity of exposition we will alternate notation $\mathcal{N}_{k_l}^{k_r}[\mathcal{L},\alpha]$ with $\mathcal{N}_{k_l}^{k_r}$  in more complicated expressions.
\end{defi}

\begin{remark}
	$\mathcal{N}_{k_l}^{k_r}$ can be also defined by the recursive formula
	$$\mathcal{N}_{k_l}^{k_l}= \ad_\mathcal{L}^{k_l}\big(\alpha\big), \quad \mathcal{N}_{k_l}^{k_r}= \ad_\mathcal{L}^{k_{r}}\big(\alpha\mathcal{N}_{k_l}^{k_{r-1}}\big), \quad \text{for} \quad r>l \quad \text{and} \quad \mathcal{N}_{k_{l}}^{k_r} = {\rm Id}, \quad  l>r.$$
\end{remark}

\begin{defi} \label{F}
	Let us make Assumption \ref{stronger_assumption}. Having  operator $\mathcal{N}_{k_l}^{k_r}[\mathcal{L},\alpha]$ introduced in Definition \ref{N} we define function $\mathcal{F}_\ell[\mathcal{L},\alpha]$, independent of $t$, which satisfy the following recursive relationship
	\begin{eqnarray} 
		&&\mathcal{F}_0[\mathcal{L},\alpha] = {\rm Id},\nonumber\\
		&&\mathcal{F}_{\ell}[\mathcal{L},\alpha] = -\sum_{m=0}^{\ell-1}\Big( a_{\ell-m}^\ell \mathcal{N}_{k_{m+1}}^{k_\ell}[\mathcal{L},\alpha]\mathcal{F}_{m}[\mathcal{L},\alpha] \Big),
	\end{eqnarray}
	
	where the coefficients $a_j^\ell \in \mathbb{Q}$, $j = 0,1,2,\ldots, \ell$ are numbers satisfy:
	\begin{equation}\label{coef_a}
		a_j^\ell = \frac{1}{\prod_{r=0}^{j-1}(j-r)^{k_{\ell-r}+1}},\quad a_0^\ell=1.
	\end{equation}	
\end{defi}
As it was in the case of $\mathcal{N}_{k_l}^{k_r}[\mathcal{L},\alpha]$ we will write $\mathcal{F}_{\ell}$  instead of $\mathcal{F}_{\ell}[\mathcal{L},\alpha]$ if only it simplifies notation and does not lead to misunderstandings. 

\begin{exam}
	To illustrate the Definition \ref{F}  we present the exact formulas for $\mathcal{F}_j, j=1,2,3$.
	\begin{eqnarray*}
		\mathcal{F}_1[\mathcal{L},\alpha]&=&-\ad_\mathcal{L}^{k_1}\big(\alpha\big)\\
		\mathcal{F}_2[\mathcal{L},\alpha]&=&-\frac{1}{2^{k_2+1}} \ad_\mathcal{L}^{k_2}\big(\alpha \ad_\mathcal{L}^{k_1}(\alpha)\big)+\ad_\mathcal{L}^{k_2}(\alpha)\ad_\mathcal{L}^{k_1}(\alpha)\\
		\mathcal{F}_3[\mathcal{L},\alpha]&=&\frac{-1 }{ 3^{k_3+1}\cdot 2^{k_2+1}} \ad_\mathcal{L}^{k_3}\big(\alpha \ad_\mathcal{L}^{k_2}(\alpha \ad_\mathcal{L}^{k_1}(\alpha))\big) +\frac{1}{ 2^{k_3+1}}\ad_\mathcal{L}^{k_3}\big(\alpha \ad_\mathcal{L}^{k_2}(\alpha)\big) \ad_\mathcal{L}^{k_1}(\alpha)\\
		&&+\frac{1}{2^{k_2+1}}\ad_\mathcal{L}^{k_3}(\alpha) \ad_\mathcal{L}^{k_2}\big(\alpha \ad_\mathcal{L}^{k_1}(\alpha)\big)- \ad_\mathcal{L}^{k_3} (\alpha)\ad_\mathcal{L}^{k_2} (\alpha)\ad_\mathcal{L}^{k_1}(\alpha)
	\end{eqnarray*}
\end{exam}

\medskip

Subsequent corollary, which can be easily proved by induction,{ will be referred to in the proof of the main theorem while calculating nested integrals with leading frequencies $\frac{1}{(\ii\omega)^{k_1+1}}$ and  $\frac{1}{(\ii\omega)^{k_2+1}}$, where $k_1, k_2$ will be inputs of multi-index ${\textbf k^2}$
	\begin{cor}\label{summation}
		The following identity holds
		\begin{eqnarray}
			\sum_{k_1=0}^{n}\frac{1}{(\ii\omega)^{k_1+1}}a_{k_1}\sum_{k_2=0}^{n-k_1-1}\frac{1}{(\ii\omega)^{k_2+1}}b_{k_2} = \sum_{k_1+k_2=0}^{n-1}\frac{1}{(\ii\omega)^{k_1+k_2+2}}a_{k_1}b_{k_2}
		\end{eqnarray}
		with assumption that $\sum_{k=0}^{-1}a_k = 0$.
	\end{cor}
	
	\begin{thm}\label{integral_AB}
		Let us make Assumption \ref{stronger_assumption}. The highly oscillatory integral $T^d\ee^{t\mathcal{L}}u_0$ defined with (\ref{T}) can be presented as a sum of finite series $\mathcal{S}^{d}_n[\mathcal{L},\alpha](t)$ 
		and the remaining term $E^{d}_n[\mathcal{L},\alpha](t) $
		
		\begin{equation}\label{needed_T}
			T^d\ee^{t\mathcal{L}}u_0= \mathcal{S}^{d}_n[\mathcal{L},\alpha](t)+ E^{d}_n[\mathcal{L},\alpha](t),\quad\ n\geq d, \quad {\rm where} 
		\end{equation}
		
		\begin{align}\label{integral_series}
			\mathcal{S}^{d}_n[\mathcal{L},\alpha](t) &= \sum_{|{\bf k^{d}}|=0 }^{n-d}\frac{1}{(\ii\omega)^{d+|{\mathbf k^d}|}}\sum_{\ell=0}^{d}a_{d-\ell}^d\ee^{(d-\ell)\ii\omega  t}\mathcal{N}_{k_{\ell+1}}^{k_d}\ee^{t\mathcal{L}}\mathcal{F}_{\ell}u_0,\\\nonumber
			&\\ \label{integral_series_error}
			E^{d}_n[\mathcal{L},\alpha](t) &= \frac{1}{(\ii\omega)^{n}}\sum_{|{\mathbf k^d}|=n-d+1}^{}\ \ \sum_{\ell=0}^{d-1}b_{d-\ell}^d \int_{0}^{t}\ee^{{(t-\tau_1)\mathcal{L}}}\mathcal{N}_{k_{\ell+1}}^{k_d}\ee^{\tau_1\mathcal{L}}\ee^{(d-\ell)\ii\omega \tau_1}\D \tau_1 \mathcal{F}_{\ell}u_0\\\nonumber
			&+ \int_{0}^{t}\ee^{(t-\tau_1)\mathcal{L}}\alpha\ee^{\ii\omega \tau_1}E^{d-1}_n[\mathcal{L},\alpha](\tau_1) u_0\D \tau_1,\\	\nonumber
		\end{align}
		and $E_n^0[\mathcal{L},\alpha](t)\equiv 0 $ while $E_n^{d-1}[\mathcal{L},\alpha](t)$ is the remaining term of the analogous integral $T^{d-1}\ee^{t\mathcal{L}}u_0$ and coefficients 
		$b_j^\ell$, $j = 0,1,2,\ldots, \ell$ satisfy:
		\begin{equation}\label{coef_b}
			b_j^\ell = 
			\frac{1}{j^{k_\ell}\cdot\prod_{r=1}^{j-1}(j-r)^{k_{\ell-r}+1}}, \quad b_0^\ell=1, 
		\end{equation}		
		with $\prod_{r=1}^{0} = 1$. Moreover 
		\begin{equation}\label{T_estimate}
			\|T^d\|_Y\leq\omega^{-d}C_T^d(t^\star,\mathcal{L},u_0,\alpha),
		\end{equation}
		were  $C_T^d(t^\star,\mathcal{L},u_0,\alpha$) is a constant depending on $d$, $t^\star$, operator $\mathcal{L}$, initial condition $u_0$ and function $\alpha$.
	\end{thm}

	\begin{remark}
		To shorten notation will write $S_n^d(t)$ and $E_n^d(t)$ instead of $S_n^d[\mathcal{L},\alpha](t)$ and  $E_n^d[\mathcal{L},\alpha](t)$ respectively whenever it does not lead to misunderstanding. We will write also $T^d\ee^{t\mathcal{L}}$ instead of $T^d\ee^{t\mathcal{L}}u_0$. \\
	\end{remark}

	\begin{proof}\kern-1em{\it of Theorem \ref{integral_AB}} 
		
		We will prove (\ref{needed_T}) by induction on $d$. Thus taking $d=1$ we consider the single integral $T^1\ee^{t\mathcal{L}}=\int_{0}^{t}\ee^{\ii\omega\tau}\ee^{(t-\tau)\mathcal{L}}\alpha\ee^{\tau\mathcal{L}}\D\tau$. In this simplest case  multi-index ${\mathbf k^1} \in \mathbb{N}^1$ consists of one input $k_1$ only, and $a_1^1 = b_1^1=1$ while $\mathcal{F}_1=-\ad_\mathcal{L}^{k_1}\big(\alpha\big)$. With these inputs the considered integral  is given with the following formula
		
		\begin{eqnarray*}
			&&T^1\ee^{t\mathcal{L}} = \sum_{k_1=0}^{n-1}\frac{1}{(\ii\omega)^{1+k_1}}\left(\ee^{\ii\omega t}\ad_\mathcal{L}^{k_1}(\alpha)\ee^{t\mathcal{L}}+\ee^{t\mathcal{L}}\left(-\ad_\mathcal{L}^{k_1}(\alpha)\right)\right)+\frac{1}{(\ii\omega)^n}\int_{0}^{t}\ee^{(t-\tau_1)\mathcal{L}}\ad_\mathcal{L}^{n}(\alpha)\ee^{\tau_1\mathcal{L} }\ee^{\ii\omega \tau_1}\D \tau_1,
		\end{eqnarray*}
		which fully coincides with (\ref{IA0B}) from Lemma \ref{IA0BLemma}  and proves the first step of induction. 
		Suppose now that the formula is valid for $T^{d}\ee^{t\mathcal{L}}$. Then 
		\begin{eqnarray*}
			&&T^{d+1}\ee^{t\mathcal{L}}=\\
			&& \int_{0}^{t} \ee^{(t-\tau_1)\mathcal{L}}\alpha\ee^{\ii\omega \tau_1}T^{d}\ee^{\tau_1\mathcal{L}} \D \tau_1\stackrel{(1)}{=} \int_{0}^{t} \ee^{(t-\tau_1)\mathcal{L}}\alpha\ee^{\ii\omega \tau_1}\Big(\mathcal{S}^{d}_n(\tau_1)+ E^{d}_n(\tau_1)\Big)\D \tau_1 \\
			&\stackrel{(2)}{=}&\sum_{|{\bf k^d}|=0 }^{n-d}\frac{1}{(\ii\omega)^{d+|{\bf k^d}|}}\sum_{\ell=0}^{d}a_{d-\ell}^{d}\int_{0}^{t}\ee^{(t-\tau_1)\mathcal{L}}\alpha\ee^{(d-\ell+1)\ii\omega  \tau_1}\mathcal{N}_{k_{\ell+1}}^{k_d}\ee^{\tau_1\mathcal{L}}\mathcal{F}_{\ell}\D \tau_1+ \int_{0}^{t} \ee^{(t-\tau_1)\mathcal{L}}\alpha\ee^{\ii\omega \tau_1} E^{d}_{n}(\tau_1)\D \tau_1\\
			&\stackrel{(3)}{=}&\sum_{|{\bf k^d}|=0 }^{n-d}\frac{1}{(\ii\omega)^{d+|{\bf k^d}|}}\sum_{k_{d+1}=0}^{n-d-|{\bf k^d}|-1}\frac{1}{(\ii\omega)^{k_{d+1}+1}}\sum_{\ell=0}^{d}a_{d-\ell}^{d}\frac{1}{(d-\ell+1)^{k_{d+1}+1}}\cdot\\
			&&\hspace{4cm}\cdot \Bigg(\ee^{(d-\ell+1)\ii\omega  t}\ad_\mathcal{L}^{k_{d+1}}\big(\alpha\mathcal{N}_{k_{\ell+1}}^{k_d}\big)\ee^{t\mathcal{L}}\mathcal{F}_\ell-\ee^{t\mathcal{L}}\ad_\mathcal{L}^{k_{d+1}}\big(\alpha\mathcal{N}_{k_{\ell+1}}^{k_d}\big)\mathcal{F}_\ell\Bigg)\\
			&+&\sum_{|{\bf k^d}|=0 }^{n-d}\frac{1}{(\ii\omega)^{d+|{\bf k^d}|}}\frac{1}{(\ii\omega)^{n-d-|{\bf k^d}|}}\sum_{\ell=0}^{d}a_{d-\ell}^{d}\frac{1}{(d-\ell+1)^{n-d-|{\bf k^d}|}}\int_{0}^{t}\ee^{(t-\tau_1)\mathcal{L}}\ad_\mathcal{L}^{n-d-|{\bf k^d}|}\big(\alpha\mathcal{N}_{k_{\ell+1}}^{k_d}
			\big)\ee^{\tau_1\mathcal{L}}\ee^{(d-\ell+1)\ii\omega\tau_1}\D \tau_1\mathcal{F}_\ell\\
			&&\quad + \int_{0}^{t} \ee^{(t-\tau_1)\mathcal{L}}\alpha\ee^{\ii\omega \tau_1} E^{d}_n(\tau_1)\D \tau_1\\
			&\stackrel{(4)}{=}&\sum_{|{\bf k^{d+1}}|=0}^{n-d-1}\frac{1}{(\ii\omega)^{d+1+|{\bf k^{d+1}}|}}\sum_{\ell=0}^{d}a_{d+1-\ell}^{d+1}\Bigg(\ee^{(d-\ell+1)\ii\omega t}\mathcal{N}_{k_{\ell+1}}^{k_{d+1}}\ee^{t\mathcal{L}}\mathcal{F}_\ell-\ee^{t\mathcal{L}}\mathcal{N}_{k_{\ell+1}}^{k_{d+1}}\mathcal{F}_\ell\Bigg)\\
			&&\quad +\frac{1}{(\ii\omega)^{n}}\sum_{|{\bf k^{d+1}}|=n-d }^{}\sum_{\ell=0}^{d}b_{d+1-\ell}^{d+1}\int_{0}^{t}\ee^{(t-\tau_1)\mathcal{L}}\mathcal{N}_{k_{\ell+1}}^{k_{d+1}}\ee^{\tau_1\mathcal{L}}\ee^{(d-\ell+1)\ii\omega\tau_1}\D \tau_1\mathcal{F}_\ell+ \int_{0}^{t} \ee^{(t-\tau_1)\mathcal{L}}\alpha\ee^{\ii\omega \tau_1} E^{d}_n(\tau_1)\D \tau_1\\
			&\stackrel{(5)}{=}&\sum_{|{\bf k^{d+1}}|=0}^{n-d-1}\frac{1}{(\ii\omega)^{d+1+|{\bf k^{d+1}}|}}\sum_{\ell=0}^{d+1}a_{d+1-\ell}^{d+1}\ee^{(d-\ell+1)\ii\omega t}\mathcal{N}_{k_{\ell+1}}^{k_{d+1}}\ee^{t\mathcal{L}}\mathcal{F}_\ell\\
			&&\quad +\frac{1}{(\ii\omega)^{n}}\sum_{|{\bf k^{d+1}}|=n-d }^{}\sum_{\ell=0}^{d}b_{d+1-\ell}^{d+1}\int_{0}^{t}\ee^{(t-\tau_1)\mathcal{L}}\mathcal{N}_{k_{\ell+1}}^{k_{d+1}}\ee^{\tau_1\mathcal{L}}\ee^{(d-\ell+1)\ii\omega\tau_1}\D \tau_1\mathcal{F}_\ell+ \int_{0}^{t} \ee^{(t-\tau_1)\mathcal{L}}\alpha E^{d}_{n}(\tau_1)\D \tau_1\\
			&\stackrel{(6)}{=}&\mathcal{S}^{d+1}_{n}(t)+ E^{d+1}_n(t)
		\end{eqnarray*}
		In (1) we exploit inductive assumption, in (2) formula for $\mathcal{S}^{d}_{n}(t)$, in (3) we refer to Lemma \ref{IA0BLemma} and perform integrating by parts $n-d-|{\bf k_d}|$ times. Equivalence (4) is valid due to Corollary  \ref{summation}, Definition \ref{N} and formulas (\ref{coef_a}), (\ref{coef_b})  where coefficients $a_j^\ell,\ b_j^\ell$ are defined.  Definition \ref{F} is utilised in (5), more specifically we observe that $\mathcal{F}_{d+1}=-\sum_{\ell=0}^{d}a_{d+1-\ell}^{d+1}\mathcal{N}_{\ell+1}^{d+1}\mathcal{F}_\ell$.  Expressions preceding equivalence (6) coincide with (\ref{needed_T}) which proves the first thesis of the theorem.
		
		To prove (\ref{T_estimate}) we start with the observation that there exists such a constant $\bar{C}_n^d(t^\star,\mathcal{L},u_0,\alpha)$ that
		$\|\mathcal{S}^{d}_n[\mathcal{L},\alpha]\|_Y\leq \omega^{-d}\bar{C}_n^d(t^\star,\mathcal{L},u_0,\alpha)$. By induction on $n$ we can easily prove that  there exists also such a constant $\tilde{C}_n^d(t^\star,\mathcal{L},u_0,\alpha)$ that
		$\|E^{d}_n[\mathcal{L},\alpha]\|_Y\leq \omega^{-n}\tilde{C}_n^d(t^\star,\mathcal{L},u_0,\alpha)$. Since $n\geq d$ we obtain
		$$
		\|T^d\ee^{t\mathcal{L}}u_0\|_Y=\| \mathcal{S}^{d}_n[\mathcal{L},\alpha](t)+ E^{d}_n[\mathcal{L},\alpha](t)\|_Y\leq \omega^{-d}\Big( \bar{C}_n^d(t^\star,\mathcal{L},u_0,\alpha)+\tilde{C}_n^d(t^\star,\mathcal{L},u_0,\alpha)\Big),
		$$
		which proves (\ref{T_estimate}) and completes the proof.
	\end{proof}
	
	\begin{remark}\label{oscillatory_integral}
		In a more general settings, but without specifying the terms of asymptotic series and error related to the truncation, multivariate highly oscillatory integrals were investigated for example in \cite{IN_2006}, \cite{S_2008}. Authors of these manuscripts shown that
		$$
		\int_{\Omega}^{}f(\textbf{y})\ee^{\ii\omega g(\textbf{y})} \D V \sim \mathcal{O}(\omega^{-d}),
		$$
		for $\Omega \subset \mathbb{R}^d$  having a piecewise smooth boundary where $\nabla g \notin\bar{\Omega}$ and $\nabla g \not\perp \partial{\Omega}$. Also the non-resonance condition is assumed.
	\end{remark}
	
	\begin{exam}
		To illustrate Theorem \ref{integral_AB} we will present the first three integrals, namely $T^1\ee^{t\mathcal{L}}$, $T^2\ee^{t\mathcal{L}}$ and $T^3\ee^{t\mathcal{L}}$. 
		
		\begin{eqnarray*}
			T^1\ee^{t\mathcal{L}}&=& \int_{0}^{t}\ee^{(t-\tau_1)\mathcal{L}}\alpha\ee^{\tau_1\mathcal{L}}\ee^{\ii\omega \tau_1}\D \tau_1 = \mathcal{S}^{1}_{n}(t)+ E^{1}_{n}(t)\\
			\mathcal{S}^{1}_{n}(t)&=& \sum_{k_1=0}^{n-1}\frac{1}{(\ii\omega)^{k_1+1}}\Big[\ee^{\ii\omega t}\ad_\mathcal{L}^{k_1}\big(\alpha\big)\ee^{t\mathcal{L}}+\ee^{t\mathcal{L}}\big(\underbrace{-\ad_\mathcal{L}^{k_1}\big(\alpha\big)}_{\mathcal{F}_1}\big)\Big]\\
			E^{1}_{n}(t)&=&\frac{1}{(\ii\omega)^{n}}\int_{0}^{t}\ee^{\ii\omega \tau_1}\ee^{(t-\tau_1)\mathcal{L}}\ad_\mathcal{L}^{n}\big(\alpha\big)\ee^{ \tau_1\mathcal{L}}\D t\tau_1\\
		\end{eqnarray*}

		\begin{eqnarray*}
			T^2\ee^{t\mathcal{L}}&=&\int_{0}^{t}\ee^{(t-\tau_1)\mathcal{L}}\alpha\ee^{\ii\omega \tau_1} \int_{0}^{\tau_1}\ee^{(\tau_1-\tau_2)\mathcal{L}}\alpha\ee^{\tau_2\mathcal{L}}\ee^{\ii\omega \tau_2}\D \tau_2\D \tau_1= \mathcal{S}^{2}_{n}(t)+ E^{2}_{n}(t)\\
			\mathcal{S}^{2}_n(t)&=&\sum_{|{\bf k^2}|=0 }^{n-2}\frac{1}{(\ii\omega)^{|{\bf k^2}|+2}}\Bigg[\frac{1}{2^{k_2+1}}\ee^{2\ii\omega t} \ad_\mathcal{L}^{k_2}\big(\alpha \ad_\mathcal{L}^{k_1}(\alpha)\big)\ee^{t\mathcal{L}}+\ee^{\ii\omega t}\ad_\mathcal{L}^{k_2}(\alpha)\ee^{t\mathcal{L}}\big(\underbrace{-\ad_\mathcal{L}^{k_1}(\alpha)}_{\mathcal{F}_1}\big)\\
			&&\hspace{2.4cm}+\ee^{t\mathcal{L}}\Big(\underbrace{-\frac{1}{2^{k_2+1}} \ad_\mathcal{L}^{k_2}\big(\alpha \ad_\mathcal{L}^{k_1}(\alpha)\big)+\ad_\mathcal{L}^{k_2}(\alpha)\ad_\mathcal{L}^{k_1}(\alpha)}_{\mathcal{F}_2}\Big)\Bigg]\\
			E^{2}_n(t)&=&\frac{1}{(\ii\omega)^{n}}\sum_{|{\bf k^2}|=n-1}^{}\Bigg[\frac{1}{2^{k_2}}\int_{0}^{t}\ee^{(t-\tau_1)\mathcal{L}}\ad_\mathcal{L}^{k_2}\big(\alpha \ad_\mathcal{L}^{k_1}(\alpha)\big)\ee^{\tau_1\mathcal{L}}\ee^{2\ii\omega \tau_1}\D \tau_1\\
			&&\hspace{2.4cm}+\int_{0}^{t}\ee^{(t-\tau_1)\mathcal{L}}\ad_\mathcal{L}^{k_2}(\alpha)\ee^{\tau_1\mathcal{L}}\Big(-\ad_\mathcal{L}^{k_1}(\alpha)\Big)\ee^{\ii\omega \tau_1}\D \tau_1\Bigg]\\
			&&+\int_{0}^{t}\ee^{(t-\tau_1)\mathcal{L}}\alpha\ee^{\ii\omega\tau_1}E^{1}_{n}(\tau_1)\D \tau_1
		\end{eqnarray*}

		\begin{eqnarray*}
			T^3\ee^{t\mathcal{L}}&=&\int_{0}^{t}\ee^{(t-\tau_1)\mathcal{L}}\alpha\ee^{\ii\omega \tau_1}\int_{0}^{\tau_1}\ee^{(\tau_1-\tau_2)\mathcal{L}}\alpha\ee^{\ii\omega \tau_2}\int_{0}^{\tau_2}\ee^{(\tau_2-\tau_3)\mathcal{L}}\alpha\ee^{\tau_3\mathcal{L}}\ee^{\ii\omega \tau_3}\D \tau_3 \D \tau_2  \D \tau_1= \mathcal{S}^{3}_{n}(t)+ E^{3}_n(t)\\
			\mathcal{S}^{3}_{n}(t)&=&\sum_{|{\bf k^3}| =0}^{n-3}\frac{1}{(\ii \omega)^{|{\bf k^3}|+3}}\Bigg(\frac{1 }{ 3^{k_3+1}\cdot 2^{k_2+1}} \ee^{3\ii\omega t}\ad_\mathcal{L}^{k_3}\big(\alpha \ad_\mathcal{L}^{k_2}(\alpha \ad_\mathcal{L}^{k_1}(\alpha))\big) \ee^{t\mathcal{L}}+\\
			&&\hspace{1cm} +\frac{1}{ 2^{k_3+1}}\ee^{2\ii\omega t}\ad_\mathcal{L}^{k_3}\big(\alpha \ad_\mathcal{L}^{k_2}(\alpha)\big)\ee^{t\mathcal{L}}\Big(\underbrace{- \ad_\mathcal{L}^{k_1}(\alpha)}_{\mathcal{F}_1}\Big)\\
			&&\hspace{1cm}+\ee^{\ii\omega t}\ad_\mathcal{L}^{k_3}(\alpha)\ee^{t\mathcal{L}}\Big(\underbrace{-\frac{1}{2^{k_2+1}} \ad_\mathcal{L}^{k_2}\big(\alpha \ad_\mathcal{L}^{k_1}(\alpha)\big)+ \ad_\mathcal{L}^{k_2} (\alpha(x))\ad_\mathcal{L}^{k_1}(\alpha)}_{\mathcal{F}_2}\Big)\\
			&&\hspace{1cm}+\ee^{t\mathcal{L}}\Big(\underbrace{\frac{-1 }{ 3^{k_3+1}\cdot 2^{k_2+1}} \ad_\mathcal{L}^{k_3}\big(\alpha \ad_\mathcal{L}^{k_2}(\alpha \ad_\mathcal{L}^{k_1}(\alpha))\big) +\frac{1}{ 2^{k_3+1}}\ad_\mathcal{L}^{k_3}\big(\alpha \ad_\mathcal{L}^{k_2}(\alpha)\big) \ad_\mathcal{L}^{k_1}(\alpha)}\\
			&&\hspace{1cm}+\underbrace{\frac{1}{2^{k_2+1}}\ad_\mathcal{L}^{k_3}(\alpha) \ad_\mathcal{L}^{k_2}\big(\alpha \ad_\mathcal{L}^{k_1}(\alpha)\big)- \ad_\mathcal{L}^{k_3} (\alpha)\ad_\mathcal{L}^{k_2} (\alpha)\ad_\mathcal{L}^{k_1}(\alpha)}_{\mathcal{F}_3}\Big)\Bigg)+\\
			E^{3}_{n}(t)&=&\frac{1}{(\ii\omega)^{n}}\sum_{|{\bf k^3}|=n-2}^{}\Bigg(\frac{1}{3^{k_3}}\frac{1}{2^{k_2+1}}\int_{0}^{t}\ee^{(t-\tau_1)\mathcal{L}}\ad_\mathcal{L}^{k_3}\big(\alpha \ad_\mathcal{L}^{k_2}(\alpha \ad_\mathcal{L}^{k_1}(\alpha))\big)\ee^{\tau_1\mathcal{L}} \ee^{3\ii\omega \tau_1}\D \tau_1\\
			&&+\frac{1}{2^{k_3}}\int_{0}^{t}\ee^{(t-\tau_1)\mathcal{L}}\ad_\mathcal{L}^{k_3}\big(\alpha \ad_\mathcal{L}^{k_2}(\alpha)\big)\ee^{\tau_1\mathcal{L}} \Big(-\ad_\mathcal{L}^{k_1}(\alpha)\Big)\ee^{2\ii\omega \tau_1} \D \tau_1\\
			&&+\int_{0}^{t}\ee^{(t-\tau_1)\mathcal{L}}\ad_\mathcal{L}^{k_3}(\alpha)\ee^{\tau_1\mathcal{L}}\ee^{\ii\omega \tau_1}\Big(-\frac{1}{2^{k_2+1}} \ad_\mathcal{L}^{k_2}\big(\alpha \ad_\mathcal{L}^{k_1}(\alpha)\big)+ \ad_\mathcal{L}^{k_2} (\alpha)\ad_\mathcal{L}^{k_1}(\alpha)\Big) \D \tau_1\Bigg)\\
			&&+\int_{0}^{t}\ee^{(t-\tau_1)\mathcal{L}}\alpha\ee^{\ii\omega \tau_1}E^{2}_{n}(\tau_1)\D \tau_1
		\end{eqnarray*}
	\end{exam}

	%-------------------------------Main_Result----------------------------

	\section{The solution  as a modulated Fourier expansion}\label{Main_result}
	
	In this Section we will formulate the main result on the form of the asymptotic expansion for problem (\ref{eq:1.5}) and on the accuracy of its truncation. Assumption \ref{stronger_assumption} is sufficient for the derivation of the coefficients of magnitude $\omega^{-r}$, $r=0,\ldots,n$ and for the proof that the error $R_n(t)$ scales like $\omega^{-n}$. To obtain error $R_n(t)$ scaling like $\omega^{-(n+1)}$ a stronger regularity of the problem is required, because one more integration by parts will be necessary, so let us start with the formulation of the assumption.

	\begin{assumption}
		\label{strongest_assumption}
		We assume that problem (\ref{eq:1.5}) equipped with suitable boundary conditions is well posed and seek  solutions $u[t]:=u(\cdot,t) \in C^{n+1}([0,T ],H^{(n+1)p}(\Omega))$ while  $u_0\in H^{(n+1)p}(\Omega)$, where $n\in\mathbb{N}$ is a fixed number. Moreover we assume that the linear, differential operator  $\mathcal{L}: D(\mathcal{L}) \supseteq  H^{(n+1)p}(\Omega) \rightarrow L^2(\Omega)$ defined with (\ref{diff_operator}) is a generator of a differential (in strong operator topology) semigroup $e^{t\mathcal{L}}: L^2(\Omega)\rightarrow L^2(\Omega)$, that  
		$\alpha \in W^{(n+1)p,\infty}(\Omega)$ 
		and that the functions $ a_{\textbf{p}}\in C_b^\infty(\Omega)$, $|\textbf{p}|\leq p$, meaning that their all derivatives are bounded on $\Omega$
	\end{assumption}

	\begin{thm}\label{Uas}
		Let us make  Assumption \ref{strongest_assumption}.  The solution of  (\ref{eq:1.5}) can be presented as a modulated Fourier series of  form (\ref{partial_expansion}) where 
		\begin{equation}\label{coefficients_p}
			p_{0,0}(t) = \ee^{t\mathcal{L}}u_0,\quad 
			p_{r,s}(t)  =
			\sum_{d=s}^{r}\sum_{|{\bf k^{d}}|=r-d } \frac{a_s^d}{(\ii)^r}\mathcal{N}_{k_{d-s+1}}^{k_d}\ee^{t\mathcal{L}}\mathcal{F}_{d-s}u_0,\quad r\neq 0,\quad {\rm with} \quad \sum_{\mathbf{|k^0|}} = 0,
		\end{equation}
		and
		\begin{equation}\label{error_R}
			R_n(t)=\sum_{d=1}^{n}E^{d}_n[\mathcal{L},\alpha](t)+\sum_{d=n+1}^{\infty}\Big(\mathcal{S}^{d}_d[\mathcal{L},\alpha](t)+E^{d}_d[\mathcal{L},\alpha](t)\Big),
		\end{equation}
		with $S_n^d[\mathcal{L},\alpha](t)$ and $E_n^d[\mathcal{L},\alpha](t)$, $n\geq d$,  defined as (\ref{integral_series}) and (\ref{integral_series_error}), respectively. 
		Moreover $R_n(t)$ satisfies the following relation
		\begin{equation}\label{R_estimate}
			\|R_n\|_Y\leq\omega^{-(n+1)}C_R^n(t^\star,\mathcal{L},u_0,\alpha),
		\end{equation}
		were  $C_R^n(t^\star,\mathcal{L},u_0,\alpha$) is a constant depending on $n$, $t^\star$, operator $\mathcal{L}$, initial condition $u_0$ and function $\alpha$.
	\end{thm}
	
	\begin{proof} Based on the Theorem \ref{Neumann_f(x,t)} and its proof, see equations (\ref{u_star}) and (\ref{difference}), we conclude that for $f(x,t)=\alpha(x)\ee^{i\omega t}$ solution $u^\star$ of (\ref{eq:1.5}) satisfies
		\begin{equation}\label{solution}
			u^\star[t]=\sum_{d=0}^{n} T^d\ee^{t\mathcal{L}}u_0+T^{n+1}\sum_{d=0}^{\infty} T^d\ee^{t\mathcal{L}}u_0,
		\end{equation}
		where $T^d\ee^{t\mathcal{L}}u_0$ is defined with (\ref{T}). For the completeness of exposition let us rewrite the thesis of Theorem \ref{integral_AB}
		\begin{equation}\label{integral_AB_formula}
			T^d\ee^{t\mathcal{L}}u_0=\mathcal{S}_n^d[\mathcal{L},\alpha](t)+E_n^d[\mathcal{L},\alpha](t),
		\end{equation}
		and sort the summands of $\mathcal{S}_n^d[\mathcal{L},\alpha](t)$  according to their magnitude $\frac{1}{\omega^r}$ and frequency $\ee^{s\ii\omega t}$ as follows

		\begin{equation}\label{ordered}
			\mathcal{S}_n^d[\mathcal{L},\alpha](t)= \frac{1}{\omega^d}\sum_{s=0}^{d}S_{d,s}^d(t)\ee^{s\ii\omega t}u_0 +\frac{1}{\omega^{d+1}}\sum_{s=0}^{d}S_{d+1,s}^d(t)\ee^{s\ii\omega t}u_0  +  \ldots  +\frac{1}{\omega^n}\sum_{s=0}^{d}S_{n,s}^d(t)\ee^{s\ii\omega t}u_0,\quad d =1,\dots n.
		\end{equation}
		
		Based on Theorem \ref{integral_AB} we can conclude that expression $S_{r,s}^d$ is of the following form
		\begin{equation} \label{Srs}
			S_{r,s}^d(t)=\sum_{\mathbf{|k^d|} =r-d }\frac{a_s^d}{(\ii)^r}\mathcal{N}_{k_{d-s+1}^{}}^{k_{d}}\ee^{t\mathcal{L}}\mathcal{F}_{d-s}.
		\end{equation}

		Let us write the first $n$ oscillatory elements of expansion (\ref{solution}) in terms of (\ref{integral_AB_formula}) and (\ref{Srs}) in a form of a table
		
		\arraycolsep=1pt\def\arraystretch{2.2}
		$$
		\begin{array}{ccccccc}
			T^1\ee^{t\mathcal{L}}=  & \frac{1}{\omega}\sum_{s=0}^{1}S_{1,s}^1(t)\ee^{s\ii\omega t}& +\frac{1}{\omega^2}\sum_{s=0}^{1}S_{2,s}^1(t)\ee^{s\ii\omega t} & +\frac{1}{\omega^3}\sum_{s=0}^{1}S_{3,s}^1(t)\ee^{s\ii\omega t}+ & \ldots & +\frac{1}{\omega^n}\sum_{s=0}^{1}S_{n,s}^1(t)\ee^{s\ii\omega t}&+E^{1}_{n}(t)\\
			T^2\ee^{t\mathcal{L}}=&   & \quad \frac{1}{\omega^2}\sum_{s=0}^{2}S_{2,s}^2(t)\ee^{s\ii\omega t}&   +\frac{1}{\omega^3}\sum_{s=0}^{2}S_{3,s}^2(t)\ee^{s\ii\omega t} +& \ldots & +\frac{1}{\omega^n}\sum_{s=0}^{2}S_{n,s}^2(t)\ee^{s\ii\omega t}&+E^{2}_{n}(t) \\ 
			T^3\ee^{t\mathcal{L}}= &  & & \quad\frac{1}{\omega^3}\sum_{s=0}^{3}S_{3,s}^3(t)\ee^{s\ii\omega t+} & \ldots& +\frac{1}{\omega^n}\sum_{s=0}^{3}S_{n,s}^3(t)\ee^{s\ii\omega t}&+E^{3}_{n}(t)  \\
			\vdots &   &  &   & \ddots& \vdots \\
			T^{n}\ee^{t\mathcal{L}}=& &  &   &  & \quad  \frac{1}{\omega^n}\sum_{s=0}^{n}S_{n,s}^n(t)\ee^{s\ii\omega t}&+E^{n}_{n}(t) \\  
		\end{array}
		$$
		Summing up $r$-th column of the above table we obtain
		$$\frac{1}{\omega^r}\sum_{d=0}^{r} \sum_{s=0}^{d}S_{r,s}^d\ee^{s\ii\omega t}=\frac{1}{\omega^r}\sum_{s=0}^{r}\ee^{s\ii\omega t}\sum_{d=s}^{r}S_{r,s}^d,\quad {\rm where} \quad S_{r,s}^0(t)=0.
		$$

		Hence, based also on (\ref{integral_AB_formula}) and (\ref{ordered}), equation (\ref{solution}) is equivalent to
		\begin{equation}\label{asymptotic_series_p}
			\begin{split}
			&	u^\star[t]\\
				&=\ee^{t\mathcal{L}}u_0+\sum_{r=0}^{n}\frac{1}{\omega^r}\sum_{s=0}^{r}\ee^{s\ii\omega t}\sum_{d=s}^{r}S_{r,s}^d u_0 +\sum_{d=1}^nE^{d}_nu_0+\sum_{d=n+1}^\infty T^{d}\ee^{t\mathcal{L}}u_0\\
				&=\ee^{t\mathcal{L}}u_0 + \sum_{r=1}^{n}\frac{1}{\omega^r}\sum_{s=0}^{r}\ee^{s\ii\omega t}\sum_{d=s}^{r} \sum_{\mathbf{|k^d|} =r-d }\frac{a_s^d}{(\ii)^r}\mathcal{N}_{k_{d-s+1}^{}}^{k_{d}}\ee^{t\mathcal{L}}\mathcal{F}_{d-s}u_0 +\sum_{d=1}^nE^{d}_nu_0  +\sum_{d=n+1}^{\infty}\Big(\mathcal{S}^{d}_d[\mathcal{L},\alpha](t)+E^{d}_d[\mathcal{L},\alpha](t)\Big)\\ 
				&= \ee^{t\mathcal{L}}u_0 + \sum_{r=0}^{n}\frac{1}{\omega^r}\sum_{s=0}^{r}\ee^{s\ii\omega t}p_{r,s}(t)u_0+R_n(t)\\ 
				%&=p_{0,0}(t) +  \sum_{r=1}^{n}\frac{1}{\omega^r}\sum_{s=0}^{r}\ee^{s\ii \omega t}p_{r,s}(t)
			\end{split}
		\end{equation}
		
		where coefficients $p_{r,s}(t)$ and error $R_n(t)$ coincide with (\ref{coefficients_p}) and (\ref{error_R}), respectively, which proves the first part of the thesis.  
		To prove estimate (\ref{R_estimate}) let us observe that
		$$
		R_n(t)=\sum_{d=1}^nE^{d}_nu_0+T^{n+1}\sum_{d=0}^\infty T^{d}\ee^{t\mathcal{L}}u_0.
		$$
		As it was pointed out in the proof of Theorem \ref{integral_AB} $\|E^{d}_n[\mathcal{L},\alpha]\|_Y\leq \omega^{-n}\tilde{C}_n^d(t^\star,\mathcal{L},u_0,\alpha)$ which does not suffice for the desired estimate of $R_n(t)$. Assumption \ref{strongest_assumption} allows for a longer expansion of $T^d$
		\begin{equation}\nonumber
			T^d\ee^{t\mathcal{L}}u_0=\mathcal{S}_{n+1}^d[\mathcal{L},\alpha](t)+E_{n+1}^d[\mathcal{L},\alpha](t),
		\end{equation}
		where $\mathcal{S}_{n+1}^d[\mathcal{L},\alpha](t)$ and $E_{n+1}^d[\mathcal{L},\alpha](t)$ are defined as previously which means that $E^{r}_{n}(t)$ can be replaced with a sum $\frac{1}{\omega^{n+1}}\sum_{s=0}^{1}S_{n+1,s}^d(t)\ee^{s\ii\omega t}+E^{d}_{n+1}(t)$ for $d=1,\ldots n$ which is bounded in $Y$ norm by $\omega^{-(n+1)}C_E^d(t^\star,\mathcal{L},u_0,\alpha)$. 
		Based on Theorem \ref{Neumann_f(x,t)} we claim that
		$$
		\|T^{n+1}\sum_{d=0}^\infty T^{d}\ee^{t\mathcal{L}}u_0\|_Y\leq C_{t^\star}\|u_0\|_{L^2(\Omega)}\exp\Big(C_{t^\star}  \|\alpha\|_{L^2(\Omega)} {t^\star}\Big) \left\|T^{n+1}\right\|_Y,  
		$$
		where $ C_{t^\star}=\max_{0\leq t\leq {t^\star}}\|e^{t\mathcal{L}}\|_{L^2(\Omega)} $,
		while from Theorem \ref{integral_AB} we know that $$\|T^{n+1}\|_Y\leq\omega^{-(n+1)}C_T^{n+1}(t^\star,\mathcal{L},u_0,\alpha).$$ Altogether we get that
		$$
		\|R_n\|_Y \leq \omega^{-(n+1)}\Big( \sum_{d=1}^n C_E^d(t^\star,\mathcal{L},u_0,\alpha) + C_{t^\star}\|u_0\|_{L^2(\Omega)}\exp\Big(C_{t^\star}  \|\alpha\|_{L^2(\Omega)} {t^\star}C_T^{n+1}(t^\star,\mathcal{L},u_0,\alpha)\Big), 
		$$
		which completes the proof.
	\end{proof}
	
	\begin{exam}
		The first few coefficients of the derived example are given with the following formulas 
		\begin{eqnarray*}
			&&p_{0,0}(t) = \ee^{t\mathcal{L}}u_0,\\
			&& p_{1,0}(t) = -\frac{1}{\ii}\ee^{t\mathcal{L}}\alpha u_0,\\
			&&p_{1,1}(t) = \frac{1}{\ii}\alpha\ee^{t\mathcal{L}}u_0,\\
			&&p_{2,0}(t) = -\frac{1}{2}\ee^{t\mathcal{L}}\alpha^2u_0+\ee^{t\mathcal{L}}\ad_\mathcal{L}^1\left(\alpha\right)u_0,\\
			&&p_{2,1}(t) = -\ad_\mathcal{L}^1\left(\alpha\right)\ee^{t\mathcal{L}}u_0+\alpha\ee^{t\mathcal{L}}\alpha u_0,\\
			&& p_{2,2}(t) = -\frac{1}{2}\alpha^2\ee^{t\mathcal{L}}u_0;\\
			&&p_{3,0}(t) =  \ii\ee^{t\mathcal{L}}\left(-\ad_\mathcal{L}^2(\alpha)-\frac{1}{4}\ad_\mathcal{L}^1(\alpha^2)+\ad_\mathcal{L}^1(\alpha)\alpha+\frac{1}{2}\alpha \ad_\mathcal{L}^1(\alpha)-\frac{1}{6}\alpha^3\right)u_0,\\
			&& p_{3,1}(t) = \ii\left(\ad_\mathcal{L}^2(\alpha)\ee^{t\mathcal{L}} -\ad_\mathcal{L}^1(\alpha)\ee^{t\mathcal{L}}\alpha-\alpha\ee^{t\mathcal{L}}\ad_\mathcal{L}^1(\alpha)+\frac{1}{2}\alpha\ee^{t\mathcal{L}}\alpha^2\right)u_0,\\
			&& p_{3,2}(t) =\ii\left( \frac{1}{4}\ad_\mathcal{L}^1(\alpha^2)\ee^{t\mathcal{L}}+\frac{1}{2}\alpha \ad_\mathcal{L}^1(\alpha)\ee^{t\mathcal{L}}-\frac{1}{2}\alpha^2\ee^{t\mathcal{L}}\alpha\right)u_0,\\
			&& p_{3,3}(t) = \ii\frac{1}{6}\alpha^3\ee^{t\mathcal{L}}u_0.
		\end{eqnarray*}
	\end{exam}
	
	\section{Numerical examples}\label{Numerics}
	
	To illustrate the proposed approach we  analyse two problems  chosen in such a way, that their exact solutions are known. This way we present the exact errors between the analytical solutions $u^\star[t]$ and their first four approximations
	\begin{equation} \label{asymptotyczne}
		U^{[n]}:=  p_{0,0}(t) +  \sum_{r=1}^{n}\frac{1}{\omega^r}\sum_{s=0}^{r}\ee^{s\ii \omega t}p_{r,s}(t), \quad n=0,\ldots,3, \quad t\in[0,1],
	\end{equation}

	\noindent
	
	\textit{Example 1.} \\
	
	Let us start with damped heat equation equipped with initial--boundary conditions
	\begin{eqnarray}\label{heat} \nonumber
		&&\partial_tu(x,t) = \Delta u(x,t)+\ee^{\ii\omega t}\cdot u(x,t), \quad x\in [0, \pi], \quad t \in [0,1],\\
		&& u(x,0) = u_0(x) = \sin(x)\cdot\exp\left(-\frac{\ii}{\omega}\right),\\
		&& u(0,t)= u(\pi,t)=0, \nonumber
	\end{eqnarray}
	with known solution
	$$
	u^\star_1[t](x)=  \ee^{-t-\ii\ee^{\ii\omega t}/\omega}\sin(x).
	$$
	
		\begin{figure}[hbt!]
		\centering
		\begin{subfigure}[b]{0.450\textwidth}
			\centering
			\includegraphics[width=\textwidth]{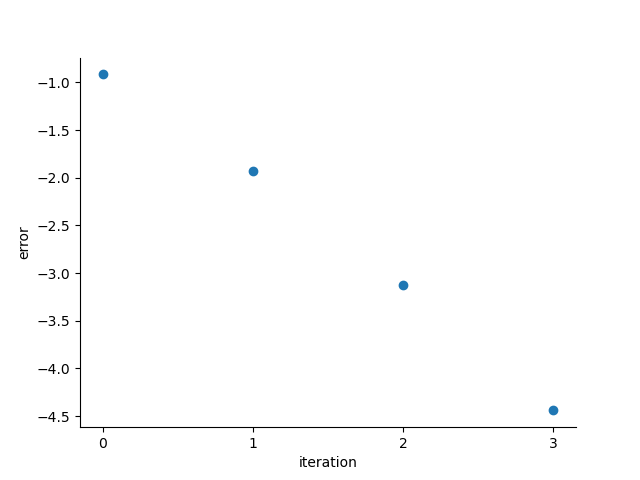}
			\caption{$\omega = 10$}
			\label{fig:omega=10}
		\end{subfigure}
		\hfill
		\begin{subfigure}[b]{0.450\textwidth}
			\centering
			\includegraphics[width=\textwidth]{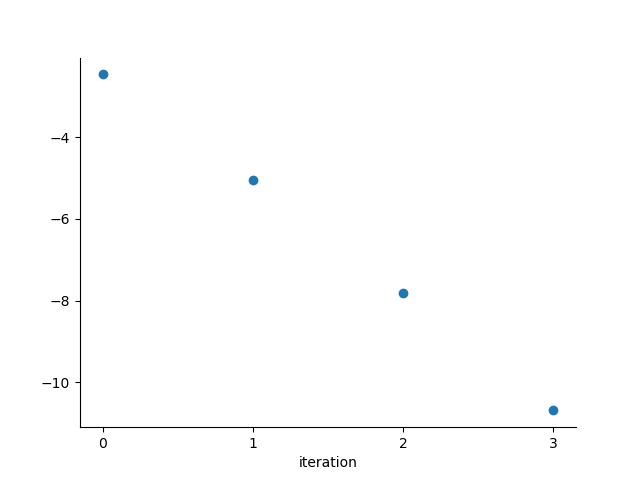}
			\caption{$\omega = 100$}
			\label{fig:omega=100}
		\end{subfigure}
		\hfill
		\caption{The base-10 logarithmic errors $R^{[n]}_1=\|u^\star_1-U^{[n]}_1\|_Y$, $n=0,\ldots,3$ obtained for the heat equation (\ref{heat}).}
		\label{fig:three graphs}
	\end{figure}
	
	Let us observe that the initial condition $u_0(x)$ and function $\alpha(x)$  do not affect error constant. Indeed, $u_0$ is non-oscillatory while $\alpha$ is constant - their derivatives are small and vanish, respectively. Thus, according to the expectations the error decays as the number of summands in (\ref{asymptotyczne}) grows. Moreover we can see that the method improves with the grow of the oscillatory parameter $\omega$. This phenomena is presented in Figure \ref{fig:omega=10} - \ref{fig:omega=100}, where errors $R^{[n]}_1:=\|u^\star_1-U^{[n]}_1\|_Y$, $n=0,\ldots,3$ are displayed in logarithmic scale.

	\newpage
	\noindent
	\textit{Example 2.} \\
	
	Let us consider now transport equation
	\begin{eqnarray}\label{transport}
		&&\partial_tu(x,t) = \nabla u(x,t)+\cos(c(x^2+x))\cdot \ee^{\ii\omega t}\cdot u(x,t),  \quad c={\rm const}, \quad x\in [-1,1], \quad t \in [0,1],\\
		&& u(x,0) = u_0(x) = x \nonumber,
	\end{eqnarray}
	with known solution
	\begin{eqnarray*}
		&&u^\star_2[t](x)= (t+x) \cdot\\
		&&  \cdot \exp \left(\frac{(-1)^{3/4} \sqrt{\pi } \exp (f_1(x,t)) \left(e^{\frac{i \left(c^2+\omega ^2\right)}{2 c}} \left(\text{erf}\left(g_1(x)\right)-\text{erf}\left(f_2(x,t)\right)\right)-\text{erfi}\left(f_3(x,t)\right)+\text{erfi}\left(g_2(x)\right)\right)}{4 \sqrt{c}}\right)
	\end{eqnarray*}
	where
	\begin{eqnarray*}
		f_1(x,t)& =& -\frac{i \left(c^2-2 c \omega  (2 t+2 x+1)+\omega ^2\right)}{4 c},\\
		g_1(x)& =&\frac{\sqrt[4]{-1} (2 c x+c+\omega )}{2 \sqrt{c}},\\
		f_2(x,t)&=&\frac{\sqrt[4]{-1} (2 c t+2 c x+c+\omega )}{2 \sqrt{c}},\\
		f_3(x,t)&=&\frac{\sqrt[4]{-1} (2 c (t+x)+c-\omega )}{2 \sqrt{c}},\\
		g_2(x)&=&\frac{\sqrt[4]{-1} (2 c x+c-\omega )}{2 \sqrt{c}},
	\end{eqnarray*}
	and  $\text{erf}(x)$,  $\text{erfi}(x)$ are the error function and the imaginary error function respectively, i.e $\text{erf}(x)= \frac{2}{\sqrt{\pi}}\int_0^x \ee^{-x^2}\D x$, $\text{erfi}(x) = -\ii\text{erf}(\ii x)$.
	In this example we wish to show how function $\alpha(x) = \cos(c(x^2+x))$ may affect the error constant. In the non-oscillatory case of $\alpha$, where $c=1$, the error  decays with number $n$ in (\ref{asymptotyczne}), see Figure \ref{fig:c=1}. An interesting situation takes place when $\alpha$ becomes oscillatory, for example $c=31$. Derivatives of $\alpha$ grow fast resulting in   $R^{[2]}_2>\omega^{-2}$ and $R^{[3]}_2>\omega^{-3}$, what can be observed in Figure  \ref{fig:c=31}.

	\begin{figure}[hbt!]
		\centering
		\begin{subfigure}[b]{0.400\textwidth}
			\centering
			\includegraphics[width=\textwidth]{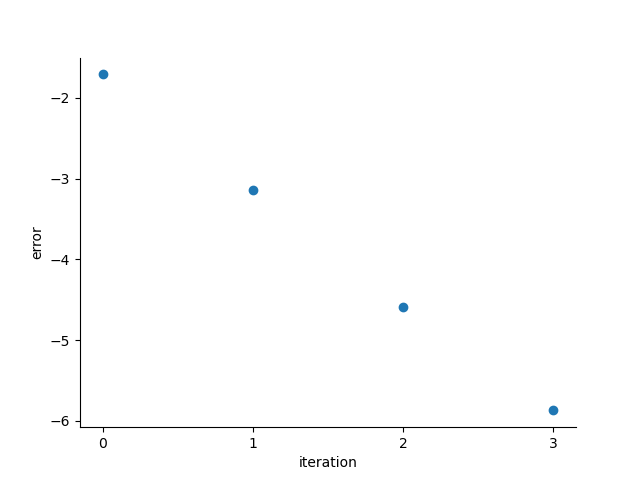}
			\caption{$c=1, \quad \omega = 100$}
			\label{fig:c=1}
		\end{subfigure}
		%\hfill
		\begin{subfigure}[b]{0.400\textwidth}
			\centering
			\includegraphics[width=\textwidth]{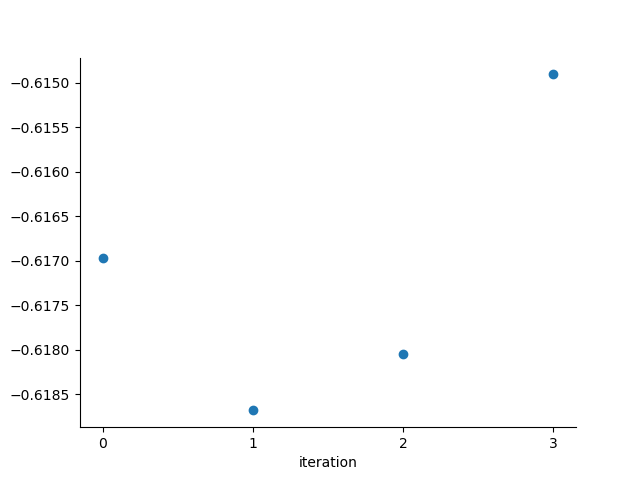}
			\caption{$c=31, \quad \omega = 100$}
			\label{fig:c=31}
		\end{subfigure}
		\hfill
		\caption{The base-10 logarithmic errors $R^{[n]}_2=\|u^\star_2-U^{[n]}_2\|_Y$, $n=0,\ldots,3$ obtained for the heat equation (\ref{transport})}
		\label{fig:three grap hs}
	\end{figure}
	
	In the  table below we present first few functions $p_{r,n}$ of both examples.
	\renewcommand{\arraystretch}{1.2}
	\begin{small}
		\begin{center}
			\begin{table} [hbt!] 
				\centering
				\begin{tabular}{ c| c| c}  
					$p_{r,s}$ & $\mathcal{L}= \nabla$, $\alpha(x) = \cos(c(x^2+x)), \  u_0(x) = x$ & $\mathcal{L}= \Delta$, $\alpha(x)\equiv 1$, $u_0(x) = \sin(x)\ee^{-\ii/\omega}$\\ 
					\hline 
					$p_{0,0}$ & $u_0(x+t)$ & $\ee^{-t}\sin(x)\ee^{-\ii/\omega}$  \\  
					$p_{1,0}$ & $-\frac{1}{\ii}\alpha(x+t)u_0(x+t)$ &     $-\frac{1}{\ii}\ee^{-t}\sin(x)\ee^{-\ii/\omega}$ \\
					$	p_{1,1}$&$\frac{1}{\ii}\alpha(x)u_0(x+t)$ &  $\frac{1}{\ii}\ee^{-t}\sin(x)\ee^{-\ii/\omega}$ \\
					$	p_{2,0}$& $-\frac{1}{2}\alpha^2(x+t)u_0(x+t)+\alpha'(x+t) u_0(x+t)$& $-\frac{1}{2}\ee^{-t}\sin(x)\ee^{-\ii/\omega}$  \\
					$	p_{2,1}$& $-u_0(x+t)\alpha'(x)+\alpha(x)\alpha(x+t)u_0(x+t) $& $\ee^{-t}\sin(x)\ee^{-\ii/\omega}$ \\
					$	p_{2,2}$& $-\frac{1}{2}\alpha^2(x)u_0(x+t)$& $-\frac{1}{2}\ee^{-t}\sin(x)\ee^{-\ii/\omega}$  \\
				\end{tabular}
				\captionof{table}{Coefficients $p_{r,n}$ of asymptotic expansion of Example 1. and Example 2.}
			\end{table}
		\end{center}\label{Tab:coefficients} 
	\end{small}

		\section{Acknowledgement}
		
		We would like to thank Arieh Iserles from the University of Cambridge and Marcus Webb from the University of Manchester  for numerous stimulating discussions which were very helpful and inspirational. \\
		
		\noindent
		We wish also to thank Anna Kamont from Institute of Mathematics Polish Academy of Sciences for her support, deep observations and rigorous approach to the contents of this paper, Jakub Skrzeczkowski from  University of Warsaw for fruitful discussions about semi-group theory and  Karolina Lademann from the University of Gda\'{n}sk for contribution to the performance of numerical examples.\\
		
		\noindent
		The work of KK in this project was supported by The National Center for Science (NCN), based on Grant No. 2019/34/E/ST1/00390. This work was also partially supported by the Simons Foundation Award No 663281 granted to the Institute of Mathematics of the Polish Academy of Sciences for the years 2021-2023. Numerical  simulations were carried out  at the Centre of Informatics Tricity Academic Supercomputer and networK (CI TASK)  in Gda\'{n}sk.

%\begin{thebibliography}
%\bibliographystyle{agsm}
\bibliographystyle{unsrt}
\bibliography{KKRP}

%\end{thebibliography}
\end{document}